\documentclass[11pt]{article}
\usepackage{authblk}
\usepackage{amssymb}
\usepackage[all]{xy}
\usepackage{mathrsfs}
\usepackage{amsmath,amssymb,amsthm}
\usepackage{extarrows}
\newtheorem{theorem}{Theorem}[section]

\newtheorem{definition}[theorem]{Definition}
\newtheorem{question}[theorem]{Question}
\newtheorem{remark}[theorem]{Remark}

\newtheorem{example}[theorem]{Example}
\newtheorem{non-example}[theorem]{Non-Example}
\newtheorem{problem}[theorem]{Problem}
\def\to{\rightarrow}
\def\ot{\otimes}
\def\f{\mathfrak}

\def\c{\mathcal}
\def\b{\mathbf}
\def\r{\mathrm}
\def\bb{\mathbb}

\date{}
\begin{document}
\title{Random Quantum Maps and their Associated \\Quantum Markov Chains}
\author{Maysam Maysami Sadr\thanks{sadr@iasbs.ac.ir, corresponding author, orcid.org/0000-0003-0747-4180}
\hspace{3mm}\&\hspace{3mm}Monireh Barzegar Ganji}
\affil{Department of Mathematics,\\ Institute for Advanced Studies in Basic Sciences (IASBS),\\ Zanjan, Iran}
\maketitle
\begin{abstract}
The notion of `quantum family of maps' (QFM) has been defined by Piotr So{\l}tan as a noncommutative analogue of
`parameterized family of continuous maps' between locally compact spaces. A QFM between C*-algebras $B,A$,
is given by a pair $(C,\phi)$ where $C$ is a C*-algebra and $\phi:B\rightarrow A\check{\otimes}C$ is a
$*$-morphism. The main goal of this note, is to introduce the notion of `random quantum map' (RQM), which is a noncommutative
analogue of `random continuous map' between compact spaces. We define a RQM between $B,A$, to be given by a triple
$(C,\phi,\nu)$ where $(C,\phi)$ is a QFM and $\nu$ a state (normalized positive linear functional) on $C$.
Our first application of RQMs takes place in theory of completely positive maps (CPM): RQMs give rise canonically to a class of CPMs which we
call implemented CPMs. We consider some partial results about the natural and important problem of characterization of implemented CPMs.
For instance, using Stinespring's Theorem, we show that any CPM from $B$ to $A$ is implemented if $A$ is finite-dimensional.
Our second application of RQMs takes place in theory of quantum stochastic processes: We show that iterations of any RQM with $B=A$,
gives rise to a quantum Markov chain in a sense introduced by Luigi Accardi.


\end{abstract}
\section{Introduction}
A random continuous map between compact spaces $X,Y$, may be defined as a probability measure on the space
of all continuous maps from $X$ to $Y$ (with some appropriate $\sigma$-algebra). More generally, if we have a continuously parameterized family
$\{\phi(.,z)\}_{z\in Z}$ of continuous maps from $X$ to $Y$, given by a continuous map $\phi:X\times Z\to Y$, then any probability
Borel measure on $Z$ may be regarded as a random continuous map. In the case $X=Y$, iterations of a random map give rise to
a Markov chain with state space $X$. Random maps, their ergodic properties, and the associated Markov chains have been studied deeply by many
authors. For instance, see Kifer's monographs \cite{Kifer1} and \cite{Kifer2}.

The main aim of this note is to introduce the notion of \emph{random quantum map} (RQM) and some of its applications.
This notion is a \emph{noncommutative} analogue of the notion of \emph{random continuous map} between compact spaces.

Piotr So{\l}tan in \cite{Soltan1} has introduced the notion of quantum family of maps (QFM) as a noncommutative analogue of the notion of
\emph{parameterized family of continuous maps} between locally compact spaces. (See also \cite{Sadr1,Sadr3} for its pure-algebraic analogue.)
A QFM from a C*-algebras $B$ to another C*-algebra $A$ (or more precisely, from the underlying noncommutative space of $A$ to that of $B$)
is given by a pair $(C,\phi)$ where $C$ is a C*-algebra, called parameter-algebra of the family, and $\phi:B\to A\check{\ot} C$
is a $*$-morphism. As it is shown in \cite{Soltan1}, one can combine QFMs; in particular, one may consider iterations
of a QFM from a C*-algebra to itself.

In this note, we define a RQM from $B$ to $A$, to be given by a triple $(C,\phi,\nu)$ where $(C,\phi)$ is a QFM from $B$ to $A$, and $\nu$ a
state (normalized positive linear functional) on $C$.
Our first application of RQMs takes place in theory of completely positive maps (CPM).
Any RQM $(C,\phi,\nu)$ as above, canonically gives rise to a completely positive linear map (CPM) $\c{F}:B\to A$ defined by
$b\mapsto(\r{id}_A\ot\nu)\phi(b)$ for every $b\in B$. In this case, we say that $\c{F}$ is \emph{implemented by} $(C,\phi,\nu)$.
Then the natural problem of characterization of implemented CPMs arises. This problem is a noncommutative analogue of the
classical problem of characterization of probability transition kernels (between compact spaces) which are induced by random continuous maps.
(See $\S$\ref{2109180602} for more details.) The classical problem and its variants have been studied
by many authors. For instance, see \cite{BlumenthalCorson2,Jost1,Rebowski1}.
Our second application of RQMs takes place in theory of quantum stochastic processes.
As iterations of random continuous maps give rise to classical Markov chains, we show that iterations of any RQM with $B=A$,
gives rise to a quantum Markov chain (QMC) in a sense introduced by Luigi Accardi \cite{Accardi0}-\cite{AccardiSouissiSoueidy1}.

In $\S$\ref{2109180602}, we review classical notions of \emph{random map} and \emph{Markov chain}.
In $\S$\ref{2109180603}, we introduce the notion of \emph{random quantum map} and consider the problem of characterization of implemented CPM.
Some of our results in this direction are listed in the following: (i) The set of implemented CPMs between any two C*-algebras is convex
(Theorem \ref{2211201043}). (ii) The set of implemented (unit-preserving) CPMs from a finitely generated C*-algebra to a
finite-dimensional C*-algebra is compact in weak operator topology (Theorem \ref{2211211700}).
(iii) Any CPM between finite-dimensional C*-algebras is implemented by a RQM with a finite-dimensional
parameter-algebra (Theorem \ref{2211170600}). (iv) Any CPM into a finite-dimensional C*-algebra is implemented (Theorem \ref{2211171200}).
(Proofs of (iii) and (iv) use Stinespring's Theorem.)
(v) There exists a (probability transition) kernel between compact spaces which is not induced by any random continuous map,
but its associated CPM is implemented (Example \ref{2211171300}).
In $\S$\ref{2109180604}, we consider QMCs associated to RQMs. In $\S$\ref{2109180605}, we consider invariant
states of RQMs and stationary QMCs.

\vspace{4mm}

\textbf{Conventions.}
Throughout, any C*-algebra has unit. \emph{*-homomorphism} is abbreviated to \emph{morphism}. Morphisms preserve units. $\b{M}_n$ denotes the
C*-algebra of $n\times n$ matrixes. For a Hilbert space $H$, the C*-algebra of bounded linear operators on $H$ is denoted by $\c{L}(H)$.
For a C*-algebra $A$ we denote by $\c{S}(A)$ and $A'$, respectively, the state space of $A$ endowed with weak*-topology and the topological
dual of $A$. Spatial tensor product of C*-algebras is denoted by $\check{\ot}$. For a compact space $X$ we denote by $\c{C}(X)$ the C*-algebra
of continuous complex functions on $X$. We have the canonical identifications $\c{C}(X)'=\c{M}(X)$ and $\c{S}(\c{C}(X))=\c{P}(X)$
where $\c{M}(X)$ is the Banach space of regular complex Borel measures on $X$ and $\c{P}(X)\subseteq\c{M}(X)$ the subset of probability measures.

\vspace{5mm}

\textbf{Acknowledgement.} The author would like to express his sincere gratitude to the anonymous referee for very valuable comments
on the early version of this manuscript.
\section{Classical Random Maps and Markov Chains}\label{2109180602}
Let $X,Y$ be compact spaces. A weak*-continuous mapping $\c{K}:X\to\c{P}(Y)$ is called a (probability transition) \emph{kernel} from $X$ to $Y$.
The kernel $\c{K}$ often is identified by the mapping $X\times\c{B}(Y)\to[0,1]$ defined by $(x,B)\mapsto(\c{K}x)(B)$ where $\c{B}(Y)$ denotes
the $\sigma$-algebra of Borel sets in $Y$. A \emph{transition} from $X$ to $Y$, is a weak*-continuous linear mapping $\c{T}:\c{M}(X)\to\c{M}(Y)$
that transforms probability measures to probability measures, and a \emph{Feller-Markov operator} (FMO) from $\c{C}(X)$ to $\c{C}(Y)$,
is a unit-preserving positive linear mapping $\c{F}:\c{C}(Y)\to\c{C}(X)$. It is well-known that FMOs are exactly unit-preserving completely
positive linear mappings between commutative C*-algebras. There are canonical one-to-one correspondences between these three classes of objects:
For a kernel $\c{K}$ the assignment $f\mapsto(x\mapsto\int_Yf\r{d}(\c{K}x))$ defines a FMO; the adjoint of any FMO is a transition; and for any
transition $\c{T}$ the assignment $x\mapsto\c{T}(\delta_x)$ defines a kernel.

Let $\c{C}(X,Y)$ denote the space of all continuous maps from $X$ to $Y$,
endowed with compact-open topology. Classically, a random continuous map from $X$ to $Y$, is defined to be a probability Borel measure
$\nu$ on $\c{C}(X,Y)$. At least in the case that $Y$ is metrizable, we may associate to $\nu$ a canonical kernel $\c{K}_\nu:X\to\c{P}(Y)$ given by
$$\c{K}_\nu(x,B):=\nu\{f\in\c{C}(X,Y):fx\in B\}\hspace{10mm}(x\in X, B\in\c{B}(Y)).$$
The following natural question and its variants have been studied by
many authors. For instance, see \cite{BlumenthalCorson2,Jost1,Rebowski1}.
\begin{question}\label{2109190601}
For which kernels $\c{K}:X\to\c{P}(Y)$ does there exist a probability Borel measures $\nu$ on $\c{C}(X,Y)$
such that $\c{K}=\c{K}_\nu$?\end{question}
Although in this note we only interested in compact spaces, but note that if $X$ is an arbitrary discrete space and
$Y$ is an arbitrary Polish space then for any kernel $\c{K}$ from $X$ to $Y$ we have $\c{K}=\c{K}_\nu$ where $\nu$
denotes the product probability measure
$\nu=\prod_{x\in X}\c{K}(x)$ on $\c{C}(X,Y)=Y^X$. (The existence of $\nu$ follows from the Kolmogorov Extension Theorem.)

In order to define the noncommutative analogue of random maps, we need the following \emph{restriction} and reformulation of the above notion.
(Indeed, the main problem within the above formulation is that, analogue of the space $\c{C}(X,Y)$ for noncommutative spaces
often does not exist or can not be described by a C*-algebra. Note that for ordinary compact spaces $X,Y$,
the space $\c{C}(X,Y)$ often is not locally compact and hence the natural function-algebra
on $\c{C}(X,Y)$ is a pro-C*-algebras \cite{Phillips1}.)
\begin{definition}\label{2211170701}
Let $X,Y$ be compact spaces. A random continuous map from $X$ to $Y$, is defined to be a triple $(Z,\phi,\nu)$ where $Z$ is a compact space
called parameter-space, $\nu$ is a probability Borel measure on $Z$, and
\begin{equation}\label{2109090650}\phi:X\times Z\to Y\end{equation}is a (jointly) continuous map.\end{definition}
In the above definition, we have interpreted $\phi$ as the continuously parameterized family $\{\phi(.,z)\}_{z\in Z}$ of continuous maps
from $X$ to $Y$. The canonical mapping $$\tilde{\phi}:Z\to\c{C}(X,Y)\hspace{10mm}z\mapsto\phi(.,z),$$
is continuous and hence $\tilde{\phi}_*\nu$ is a probability Borel measure on $\c{C}(X,Y)$ with compact support. Thus, $(Z,\phi,\nu)$
defines a unique random continuous map $\tilde{\phi}_*\nu$ in the classical sense.
The kernel, transition, and FMO associated to $(Z,\phi,\nu)$ are defined by
$$\c{K}_{\phi,\nu}:X\to\c{P}(Y),\hspace{5mm}\c{K}_{\phi,\nu}(x,B):=\nu\{z\in Z:\phi(x,z)\in B\},$$
\begin{equation}\label{2109090657}
\c{T}_{\phi,\nu}:\c{M}(X)\to\c{M}(Y),\hspace{10mm}\rho\mapsto\phi_*(\rho\times\nu),\end{equation}
\begin{equation}\label{2109090658}
\c{F}_{\phi,\nu}:\c{C}(Y)\to\c{C}(X),\hspace{10mm}f\mapsto\Big(x\mapsto\int_Zf\phi(x,z)\r{d}\nu(z)\Big).\end{equation}
\begin{definition}\label{2211170702}
Let $X,Y$ be compact spaces and let $(Z,\phi,\nu)$ be a random continuous map from $X$ to $Y$. We say that a kernel $\c{K}:X\to\c{P}(Y)$
is implemented by $(Z,\phi,\nu)$ if $\c{K}=\c{K}_{\phi,\nu}$.\end{definition}
Note that for a kernel $\c{K}$ as above, $\c{K}$ is implemented by $(Z,\phi,\nu)$ iff for its associated FMO $\c{F}$ we have $\c{F}=\c{F}_{\phi,\nu}$.
We will see that in noncommutative framework, the following \emph{restriction} of question \ref{2109190601} has more appropriate
and effective analogue:
\begin{question}\label{2109190633}
Which kernels are implemented by random continuous maps?\end{question}
The easiest answer to this question is as follows:
\begin{theorem}\label{2211170601}
For $X,Y,\c{K}$ as above, if $X$ is a finite space then $\c{K}$ is implemented.\end{theorem}
\begin{proof}We have $\c{C}(X,Y)=Y^X$. $\c{K}$ is implemented by the random continuous map $(Y^X,\phi,\prod_x(\c{K}x))$ where
$\phi:X\times Y^X\to Y$ is given by $(x,f)\mapsto f(x)$.\end{proof}
\begin{non-example}\label{2211170602}
\emph{Let $X$ be the compact interval $[0,1]$ of real line and let $Y$ be the two-point set $\{0,1\}$.
The space $\c{C}(X,Y)$ has only two points, the constant functions.
We have $\c{P}(Y)=\{x^\dag: x\in[0,1]\}$ where $x^\dag$ denotes the probability measure on $Y$ defined by $\{0\}\mapsto x$ and
$\{1\}\mapsto(1-x)$. Consider the kernel $\c{K}$ from $X$ to $Y$ given by $x\mapsto x^\dag$.
Then, it is easily seen that $\c{K}$ can not be implemented by any random continuous map. Indeed, only constant kernels $X\to\c{P}(Y)$
are implemented.}\end{non-example}
For two families $\phi_1:X_1\times Z_1\to X_2$ and $\phi_2:X_2\times Z_2\to X_3$ of continuous maps their composition $\phi_2\diamond\phi_1$
is defined to be the family
\begin{equation}\label{2109110630}
\phi_2\diamond\phi_1:X_1\times(Z_1\times Z_2)\to X_3\hspace{5mm}(x_1,z_1,z_2)\mapsto\phi_2(\phi_1(x_1,z_1),z_2).\end{equation}
of maps from $X_1$ to $X_3$. Then for random maps $(Z_i,\phi_i,\nu_i)$ ($i=1,2$) with $Z_i,\phi_i$ as above
the Chapman-Kolmogorov identity is seen as any of the following three formulas:
\begin{equation*}\c{K}_{\phi_2\diamond\phi_1,\nu_1\times\nu_2}(x_1,B_3)=
\int_{X_2}\c{K}_{\phi_2,\nu_2}(.,B_3)\r{d}\c{K}_{\phi_1,\nu_1}(x_1,.),\end{equation*}
\begin{equation*}\c{T}_{\phi_2\diamond\phi_1,\nu_1\times\nu_2}=\c{T}_{\phi_2,\nu_2}\c{T}_{\phi_1,\nu_1},
\hspace{10mm}\c{F}_{\phi_2\diamond\phi_1,\nu_1\times\nu_2}=\c{F}_{\phi_1,\nu_1}\c{F}_{\phi_2,\nu_2}.\end{equation*}
Now, we consider Markov chains associated to random maps from a compact space $X$ to itself:
Let $\sigma\in\c{P}(X)$ and let for each $n\geq1$, $(Z_n,\phi_n,\nu_n)$ be a random continuous map on $X$. Consider the $X$-valued stochastic process
$(\psi_n)_{n\geq0}$ on the probability space $$(X\times\prod_{n=1}^\infty Z_n,\sigma\times\prod_{n=1}^\infty\nu_n),$$
defined by $\psi_0(\omega)=x$ and $\psi_n(\omega)=\phi_n\diamond\cdots\diamond\phi_1(x,z_1,\ldots,z_n)$ where
$\omega=(x,z_1,\ldots)\in X\times\prod_{n=1}^\infty Z_n$.
This process is a (nonhomogeneous) Markov chain with state space $X$, initial probability $\sigma$, and probability transition kernel
$\c{K}_{\phi_{m}\diamond\cdots\diamond\phi_{n+1},\nu_{n+1}\times\cdots\times\nu_{m}}$ from step $n\geq0$ to step $m>n$.
The path-space formulation of the process is given by the push-forward of the probability measure $\sigma\times\prod_{n=1}^\infty\nu_n$
under the mapping $$X\times\prod_{n=1}^\infty Z_n\to\prod_{n=0}^\infty X_n\hspace{10mm}\omega\mapsto(\psi_0(\omega),\psi_1(\omega),\ldots).$$
The \emph{Markov property} of the process may be formulated as follows: For every $n$, there exists a unit preserving positive
linear operator $\b{E}_n:\c{C}(X)\to\c{C}(X)$ such that
\begin{equation}\label{2109150710}
(\b{E}_nf)\psi_n=\int_{Z_{n+1}}f\psi_{n+1}\r{d}\nu_{n+1}.\end{equation}
Indeed we have $\b{E}_nf(x):=\int f\phi_{n+1}(x,.)\r{d}\nu_{n+1}$. Note that (\ref{2109150710}) implies that for any $f\in\c{C}(X)$
the conditional expectation of $f\psi_{n+1}$ given the sigma-algebra generated by $\psi_0,\ldots,\psi_n$ is equal to the conditional
expectation of $f\psi_{n+1}$ given the sigma-algebra generated by $\psi_n$, and also, the expectations of the random variables $f\psi_{n+1}$ and
$(\b{E}_nf)\psi_n$ (i.e. their integrals with respect to $\sigma\times\prod_{n=1}^\infty\nu_n$) are equal.
\section{Random Quantum Maps}\label{2109180603}
The concept of \emph{quantum family of maps} has been introduced by So{\l}tan \cite{Soltan1} as a noncommutative
version of the notion given by (\ref{2109090650}):
\begin{definition}
Let $A$ and $B$ be C*-algebras. A quantum family of maps (QFM) from $B$ to $A$ (or more precisely, from the
underlying noncommutative space of $A$ to that of $B$) is given by a morphism $\phi:B\to A\check{\ot} C$ between C*-algebras. Here,
the C*-algebra $C$ is called parameter-algebra of the family. If $C$ is commutative we call the family classical.\end{definition}
Note that if the QFM $\phi$ as above, is classical then $\phi$ is completely distinguished by the family
$\{(\r{id}_A\ot\alpha)\phi\}_{\alpha\in\Delta(C)}$ of morphisms (quantum maps) from $B$ to $A$, where $\Delta(C)$ denotes
the Gelfand (character) space of $C$. In the following definition, we introduce one of the main concepts of this note.
This is a noncommutative version of the notion of \emph{random continuous map} given in Definition \ref{2211170701}.
\begin{definition}A random quantum map (RQM) from $B$ to $A$ is a triple $(C,\phi,\nu)$ where $\phi:B\to A\check{\ot} C$ is a QFM
from $B$ to $A$ with the parameter-algebra $C$, and where $\nu\in\c{S}(C)$. If $C$ is commutative then the RQM is called classical.\end{definition}
Note that if the RQM $(C,\phi,\nu)$ is classical, then $\nu$ is identified by a probability Borel measure on $\Delta(C)$; thus $\nu$ gives rise to
a random structure on the family $\{(\r{id}_A\ot\alpha)\phi\}_{\alpha\in\Delta(C)}$ of quantum maps.
In order to keep notations and terminology similar to the
classical case considered in $\S$\ref{2109180602}, we make the following definition.
\begin{definition}Any unit-preserving completely positive linear mapping $\c{F}:B\to A$ is called
a noncommutative Feller-Markov operator (NFMO). A transition $\c{T}:A'\to B'$ is the adjoint of some NFMO.\end{definition}
Note that FMOs are exactly NFMOs between commutative C*-algebras. Also,
any transition $\c{T}:A'\to B'$ is a weak*-continuous linear mapping that transforms $\c{S}(A)$ to $\c{S}(B)$.
We know that any unit-preserving positive linear map $B\to A$ is a NFMO if at least one of the $A$ and $B$ is commutative
\cite[Corollary 3.14 \& Theorem 3.16]{Skoufranis1}. It can be easily seen that the set of all NFMOs from $B$ to $A$, is a norm-closed
convex subset of the Banach space $\c{L}(B,A)$ of all bounded linear operators from $B$ to $A$.

For any RQM $(C,\phi,\nu)$ as above, the associated transition and NFMO are defined by
\begin{equation}\label{2109110655}
\c{T}_{\phi,\nu}:A'\to B'\hspace{10mm}\rho\mapsto(\rho\ot\nu)\phi,\end{equation}
\begin{equation}\label{2109110658}
\c{F}_{\phi,\nu}:B\to A\hspace{10mm}b\mapsto(\r{id}\ot\nu)\phi(b).\end{equation}
Note that $\c{T}_{\phi,\nu}$ is the adjoint of $\c{F}_{\phi,\nu}$. Also, note that (\ref{2109110655}) and (\ref{2109110658}) are the same
(\ref{2109090657}) and (\ref{2109090658}), written for general C*-algebras.

The following is the analogue of Definition \ref{2211170702}, in noncommutative framework.
\begin{definition}
Let $\c{F}:B\to A$ be a NFMO and let $(C,\phi,\nu)$ be a RQM from $B$ to $A$. We say that $\c{F}$ is implemented by $(C,\phi,\nu)$ (or,
$(C,\phi,\nu)$ implements $\c{F}$) if $\c{F}=\c{F}_{\phi,\nu}$. A NFMO is called \emph{implemented} if it is implemented by some RQM.
A NFMO is called classically implemented if it is implemented by a classical RQM.\end{definition}
Although the smaller parameter-algebras are more desirable, but if a NFMO $\c{F}$ is implemented by a RQM with parameter-algebra $C$,
and if $C$ is a C*-subalgebra of a C*-algebra $D$ (with the same unit), then there is a RQM with parameter-algebra $D$, that implements $\c{F}$.
Thus, by GNS Theorem, for every implemented NFMO $\c{F}$ there is a RQM with parameter-algebra $\c{L}(H)$
for some Hilbert space $H$, that implements $\c{F}$.

The following natural question is a noncommutative version of Question \ref{2109190633}:
\begin{question}\label{2109190644}
Which NFMOs between C*-algebras are implemented?\end{question}
Before we give some easy answers to this question, we consider the notion of \emph{composition} (\cite{Soltan1}) for QFMs:
For QFMs $\phi_1:A_2\to A_1\check{\ot} C_1$ and
$\phi_2:A_3\to A_2\check{\ot} C_2$ the dual version of (\ref{2109110630}) is given by
$$\phi_1\diamond\phi_2:A_3\to A_1\check{\ot}(C_1\check{\ot} C_2)\hspace{10mm}a_3\mapsto(\phi_1\ot\r{id})\phi_2(a_3).$$
For RQMs $(C_i,\phi_i,\nu_i)$ ($i=1,2$) with $C_i,\phi_i$ as above, the Chapman-Kolmogorov identity may be expressed by the following
two equivalent identities:
$$\c{F}_{\phi_1\diamond\phi_2,\nu_1\ot\nu_2}=\c{F}_{\phi_1,\nu_1}\c{F}_{\phi_2,\nu_2},\hspace{10mm}
\c{T}_{\phi_1\diamond\phi_2,\nu_1\ot\nu_2}=\c{T}_{\phi_2,\nu_2}\c{T}_{\phi_1,\nu_1}.$$
\begin{theorem}\label{2109210530}
The following statements hold.
\begin{enumerate}
\item[(i)] States of C*-algebras and morphisms between C*-algebras are implemented.
\item[(ii)] Composition, direct sum, and tensor product of implemented NFMOs are implemented.\end{enumerate}\end{theorem}
\begin{proof}It is well-known that states and morphisms are NFMOs. Also, composition, direct sum, and (spatial) tensor product of NFMOs, are NFMOs.
(However, these facts for implemented NFMOs may be deduced from the following proof.)

Let $\sigma\in\c{S}(A)$. Then $\sigma$ is implemented by the RQM $(A,\r{id},\sigma)$ where $\r{id}$ is considered
as the morphism $\r{id}:A\to\bb{C}\check{\ot} A$.

Let $\phi:B\to A$ be a morphism. Then $\phi$ is implemented by the RQM $(\bb{C},\phi,1)$ where $\phi$ is considered as the morphism
$\phi:B\to A\check{\ot}\bb{C}$ and where $1$ denotes the unique state of the C*-algebra $\bb{C}$.

Let $\c{F}_1:A_2\to A_1$ and $\c{F}_2:A_3\to A_2$ be NFMOs implemented respectively by $(C_1,\phi_1,\nu_1)$ and $(C_2,\phi_2,\nu_2)$.
Then $\c{F}_1\c{F}_2$ is implemented by $(C_1\check{\ot} C_2,\phi_1\diamond\phi_2,\nu_1\ot\nu_2)$.

For $i=1,2$, let $\c{F}_i:B_i\to A_i$ be a NFMO implemented by $(C_i,\phi_i,\nu_i)$. Then
$$\c{F}_1\oplus\c{F}_2:B_1\oplus B_2\to A_1\oplus A_2\hspace{5mm}\text{and}\hspace{5mm}\c{F}_1\ot\c{F}_2:B_1\check{\ot} B_2\to A_1\check{\ot} A_2$$
are implemented respectively by $(C_1\check{\ot} C_2,\phi,\nu_1\ot\nu_2)$ and $(C_1\check{\ot} C_2,\psi,\nu_1\ot\nu_2)$.
Here, $\phi$ is the composition
$$\xymatrix{B_1\oplus B_2\ar[r]^-{\phi_1\oplus\phi_2}&(A_1\check{\ot} C_1)\oplus(A_2\check{\ot} C_2)\ar[r]&
(A_1\oplus A_2)\check{\ot}(C_1\check{\ot} C_2)}$$
where the second arrow is given by $(a_1\ot c_1,a_2\ot c_2)\mapsto(a_1,0)\ot(c_1\ot1)+(0,a_2)\ot(1\ot c_2)$, and $\psi$
is the composition $$\xymatrix{B_1\check{\ot} B_2\ar[r]^-{\phi_1\ot\phi_2}&
(A_1\check{\ot} C_1)\check{\ot}(A_2\check{\ot} C_2)\ar[r]&(A_1\check{\ot} A_2)\check{\ot}(C_1\check{\ot} C_2)}$$
where the second arrow just flips the components.\end{proof}
\begin{theorem}\label{2211201043}
For any two C*-algebras $A,B$, the set of all implemented NFMOs from $B$ to $A$, is a convex subset of $\c{L}(B,A)$.\end{theorem}
\begin{proof}Let $\{\c{F}_i\}_{i=1}^n$ be a family of implemented NFMOs from $B$ to $A$, and $\{t_i\}_{i=1}^n$ nonnegative reals with
$\sum_it_i=1$. By Theorem \ref{2109210530}, $\oplus_{i=1}^n\c{F}_i$ is implemented by a RQM $(C,\phi,\nu)$. Let $\psi$ denote the composition
$$\xymatrix{B\ar[r]&\oplus_{i=1}^nB\ar[r]^-{\phi}&(\oplus_{i=1}^n A)\check{\ot} C\ar[r]&A\check{\ot}(\bb{C}^n\check{\ot} C)}$$
where the first arrow is given by $b\mapsto(b,\ldots,b)$ and the third arrow is given by
$$(a_1,\ldots,a_n)\ot c\mapsto\sum_{i=1}^na_i\ot(e_i\ot c)$$ where $e_1,\ldots,e_n$ denote the Euclidean basis of $\bb{C}^n$. Let
$\sigma\in\c{S}(\bb{C}^n)$ be defined by $e_i\mapsto t_i$. Then $\sum_{i=1}^nt_i\c{F}_i$ is implemented by $(\bb{C}^n\check{\ot} C,\psi,\sigma\ot\nu)$.
\end{proof}
The following theorem gives a characterization of all classically implemented NFMOs.
\begin{theorem}
A NFMO $\c{F}:B\to A$ is classically implemented iff there exist a compact space $Z$, a probability Borel measure $\sigma$ on $Z$, and
a jointly continuous map $f:Z\times B\to A$ (continuous with respect to the norms of $A$ and $B$), such that for every fixed
$z\in Z$, $f(z,.):B\to A$ is a morphism, and such that $$\c{F}(b):=\int_Zf(z,b)\r{d}\sigma(z),\hspace{10mm}(b\in B).$$ \end{theorem}
\begin{proof} `if' part: $\c{F}$ is implemented by $(\c{C}(Z),\phi,\sigma)$ where $\phi:B\to A\check{\ot}\c{C}(Z)$ is defied by
$b\mapsto(z\mapsto f(z,b))$ through the canonical identification of $A\check{\ot}\c{C}(Z)$ with the C*-algebra $\c{C}(Z,A)$ of
continuous maps from $Z$ to $A$ (\cite[Chapter 8]{EffrosRuan1})

`only if' part: Suppose that $\c{F}$ is implemented by $(C,\phi,\sigma)$ where $C$ is commutative. By Gelfand's Theorem, $C$ is isomorphic
to $\c{C}(\Delta(C))$ where $\Delta(C)$ denotes character space of $C$. We let $Z:=\Delta(C)$ and identify $\sigma$ with a probability Borel
measure on $Z$. Then $f$ is given by $(z,b)\mapsto(\r{id}_A\ot z)\phi(b)$.\end{proof}
It follows easily from a type of Banach-Alaoglu's Theorem that for any two C*-algebras $A,B$, such that $A$ is finite-dimensional,
the set of all morphisms from $B$ to $A$, and the set of all NFMOs from $B$ to $A$, as subsets of $\c{L}(B,A)$,
are compact in strong operator topology.
\begin{theorem}\label{2211171400}
Let $A,B,C$ be C*-algebras such that $A,C$ are finite-dimensional. Then the set of NFMOs from $B$ to $A$, which are implemented by
RQMs with the fixed parameter-algebra $C$, is compact in strong operator topology.\end{theorem}
\begin{proof}Let $(\c{F}_\lambda)_\lambda$ be a net of NFMOs from $B$ to $A$. Suppose that $\c{F}_\lambda$ is implemented by RQM
$(C,\phi_\lambda,\nu_\lambda)$. Suppose that $\c{F}_\lambda$ converges to a NFMO $\c{F}$, in strong operator topology. There are subnets
of $(\phi_\lambda)_\lambda$ and $(\nu_\lambda)_\lambda$, denoted by the same index $\lambda$, such that $\phi_\lambda$ converges to a morphism
$\phi:B\to A\check{\ot}C$, in strong operator topology, and $(\nu_\lambda)_\lambda$ converges to a state $\nu\in\c{S}(C)$,
in weak*-topology or equivalently in functional norm. For every $b\in B$, we have
$$\c{F}(b)=\lim_\lambda\c{F}_\lambda(b)=\lim_\lambda(\r{id}_A\ot\nu_\lambda)\phi_\lambda(b)=(\r{id}_A\ot\nu)\phi(b).$$
Thus, $\c{F}$ is implemented by $(C,\phi,\nu)$. Therefore, we showed that the set of NFMOs from $B$ to $A$ implemented by RQMs with the
parameter-algebra $C$, is strongly-closed in the set of all NFMOs from $B$ to $A$. The proof is complete.\end{proof}
A QFM $\f{m}:B\to A\check{\ot}\f{M}$ is called quantum family of \emph{all} maps from $B$ to $A$
\cite{Soltan1} if for any other QFM $\phi:B\to A\check{\ot} C$ there exists a unique morphism $\tilde{\phi}:\f{M}\to C$ such that
$\phi=(\r{id}\ot\tilde{\phi})\f{m}$. So{\l}tan in \cite{Soltan1} has proved that $\f{m}$ exists if $B$ is finitely generated and $A$ is finite
dimensional. (This may be interpreted as the noncommutative analogue of the fact that $\c{C}(X,Y)$ is a compact space
if $X$ is finite space and $Y$ is compact.)
Suppose that $\f{m}$ exists and $\c{F}:B\to A$ is a NFMO implemented by $(C,\phi,\nu)$. Then it is clear that $\c{F}$ is also implemented
by the RQM $(\f{M},\f{m},\nu\tilde{\phi})$.
\begin{theorem}\label{2211211700}
Let $A,B$ be two C*-algebras such that the quantum family of all maps from $B$ to $A$ exists, $\text{e.g.}$ $B$ is finitely generated and $A$ is
finite-dimensional. Then the set of all implemented NFMOs from $B$ to $A$, is convex and compact in weak operator topology.\end{theorem}
\begin{proof}Suppose that $\f{m}:B\to A\check{\ot}\f{M}$ denotes the quantum family of all maps from $B$ to $A$.
Let the mapping $\Gamma:\c{S}(\f{M})\to\c{L}(B,A)$ be defined by $\nu\mapsto(\r{id}_A\ot\nu)\f{m}$. Then the image of $\Gamma$
is exactly the set of implemented NFMOs from $B$ to $A$. Let $(\nu_\lambda)_\lambda$ be a net in $\c{S}(\f{M})$ that converges to
to $\nu$ in weak*-topology. For every $f\in A'$, the net $(f\ot\nu_\lambda)_\lambda$ in $(A\check{\ot}\f{M})'$ converges to $f\ot\nu$
in weak*-topology. Thus for every $b\in B$, the net $((\r{id}_A\ot\nu_\lambda)\f{m}(b))_\lambda$ in $A$ converges to
$(\r{id}_A\ot\nu)\f{m}(b)$ in weak-topology of $A$. Therefore, $\Gamma$ is continuous with respect to weak*-topology on $\c{S}(\f{M})$ and
weak operator topology on $\c{L}(B,A)$. This shows that the set of implemented NFMOs from $B$ to $A$ is compact in weak operator topology.
The convexity follows from Theorem \ref{2211201043} (and also can be deduced directly).\end{proof}
We need the following version of Stinespring's Theorem (\cite[Theorem 4.1]{Skoufranis1}).
\begin{theorem}\label{2211141516}
(Stinespring) Let $B$ be a C*-algebra and $H$ a Hilbert space. Suppose that
$\c{F}:B\to\c{L}(H)$ is a NFMO. Then there exist a Hilbert space $K$ that
contains $H$ as a Hilbert subspace, and a morphism $\psi:B\to\c{L}(K)$, such that
$$\c{F}(b)=\r{P}^K_H\psi(b)\r{E}^H_K\hspace{10mm}(b\in B)$$
where $\r{P}^K_H$ denotes the orthogonal projection from $K$ onto $H$, and $\r{E}^H_K$ the embedding of $H$ into $K$.
\end{theorem}
As it can be seen from the proof of Stinespring's Theorem \cite[Theorem 4.1]{Skoufranis1}, to construct the desired Hilbert space $K$, firstly,
a specific semi-inner-product is put on the algebraic tensor product vector space $B\odot H$ and then, $K$ is defined to be the Hilbert space
completion of the quotient of $B\odot H$ by the null vectors. Thus, if $B$ and $H$ are finite-dimensional (resp. separable)
then $K$ may be chosen to be finite-dimensional (resp. separable).
\begin{theorem}\label{2211160907}
Let $B$ be a separable C*-algebra and let $H$ be an infinite-dimensional separable Hilbert space. Then any
NFMO $\c{F}:B\to\c{L}(H)$ is implemented by a RQM with parameter-algebra $\b{M}_2$.\end{theorem}
\begin{proof}Let $K,\psi$ be as in Theorem \ref{2211141516} and suppose that $K$ is separable. Consider the Hilbert space $\tilde{K}=K\oplus K$
and let the morphism $\tilde{\psi}:B\to\c{L}(\tilde{K})\cong\b{M}_2(\c{L}(K))\cong\c{L}(K)\check{\ot}\b{M}_2$ be given by
$b\mapsto\psi(b)\ot1$. We have $\c{F}(b)=\r{P}^{\tilde{K}}_H\tilde{\psi}(b)\r{E}^H_{\tilde{K}}$ where $H$ is considered as a subspace of
the first component of $\tilde{K}$. Let $H^\perp$ denote the orthogonal direct-summand of $H$ in $\tilde{K}$. Then $H^\perp$ is an
infinite-dimensional separable Hilbert space. Let $\phi$ denote the composition
$$\xymatrix{B\ar[r]^-{\tilde{\psi}}&\c{L}(\tilde{K})\ar[r]&\c{L}(H\oplus H^\bot)\ar[r]&\c{L}
(H\oplus H)\ar[r]&\r{M}_2(\c{L}(H))\ar[r]&\c{L}(H)\check{\ot}\b{M}_2},$$
where the second, forth, and fifth arrows are obvious isomorphisms, and the third arrow
is the isomorphism induced by an arbitrary identification between $H^\bot$ and the second component of $H\oplus H$. Let
$\sigma\in\c{S}(\b{M}_2)$ be defined by $(m_{ij})_{2\times 2}\mapsto m_{11}$. Then $\c{F}$ is implemented by $(\b{M}_2,\phi,\sigma)$.\end{proof}
\begin{theorem}\label{2211161017}
Let $B$ be a finite-dimensional C*-algebra, and let $H$ be a finite-dimensional Hilbert space. Then any NFMO from $B$ into $\c{L}(H)$
is implemented by a RQM with the parameter-algebra $\b{M}_\ell$ such that $\ell\leq\r{dim}(B)\r{dim}(H)$.
\end{theorem}
\begin{proof}Let $\r{dim}(H)=n$, and let $K,\psi$ be as in Theorem \ref{2211141516} with
$\ell=\r{dim}(K)\leq n\r{dim}(B)$. The proof is a rewriting of the proof of Theorem \ref{2211160907}, but this time we let $\tilde{K}$
be the orthogonal direct sum of $n$ copies of $K$, and to construct $\tilde{\psi}$ and $\phi$ we use the identifications
$\c{L}(\tilde{K})\cong\c{L}(K)\check{\ot}\b{M}_n$ and $\c{L}(\tilde{K})\cong\c{L}(H)\check{\ot}\b{M}_\ell$, respectively.\end{proof}
\begin{theorem}\label{2211170600}
Any NFMO between finite-dimensional C*-algebras is implemented by a RQM with finite-dimensional parameter-algebra.\end{theorem}
\begin{proof} Let $\c{F}:B\to A$ be a NFMO between finite-dimensional C*-algebras. Suppose that $A=\oplus_{i=1}^k\b{M}_{n_i}$. Let
$\r{P}_i:A\to\b{M}_{n_i}$ denote the projection onto the $i$'th component of $A$. Then $\r{P}_i\c{F}$ is a NFMO and hence, by Theorem \ref{2211161017}, is implemented by a RQM with parameter-algebra $\b{M}_{\ell_i}$ for some $\ell_i$. By Theorem \ref{2109210530}(ii), the NFMO
$\oplus_{i=1}^k\r{P}_i\c{F}:\oplus_{i=1}^kB\to A$ is implemented by a RQM with parameter-algebra $\check{\ot}_{i=1}^k\b{M}_{\ell_i}$.
Let $\phi:B\to\oplus_{i=1}^kB$ denote the morphism $b\mapsto(b,\ldots,b)$. We have $\c{F}=[\oplus_{i=1}^k\r{P}_i\c{F}]\phi$.
Thus, $\c{F}$ is implemented by a RQM with parameter-algebra $\b{M}_{\ell_1\cdots\ell_k}$.\end{proof}
The following result is a noncommutative generalization of Theorem \ref{2211170601}.
Its proof is similar to the proofs of Theorems \ref{2211161017} and \ref{2211170600}, and omitted.
\begin{theorem}\label{2211171200}
If $B$ is an arbitrary C*-algebra and $A$ a finite-dimensional C*-algebra, then any NFMO from $B$ to $A$ is implemented.\end{theorem}
The following example shows an interesting phenomena: There are implemented FMOs which are not classically implemented.
\begin{example}\label{2211171300}
\emph{Let $X=[0,1]$, $Y=\{0,1\}$, and $\c{K}$ be as in Non-Example \ref{2211170602}. It can be checked that the FMO $\c{F}$ associated to
the kernel $\c{K}$ is given by
$$\c{F}:\c{C}(Y)\cong\bb{C}^2\to\c{C}([0,1]),\hspace{10mm}(r,s)\mapsto\big[x\mapsto rx+s(1-x)\big],\hspace{5mm}(r,s\in\bb{C},x\in[0,1]).$$
By Non-Example \ref{2211170602}, we know that $\c{F}$ is not classically implemented. But, $\c{F}$ is implemented by a RQM with
parameter-algebra $\b{M}_2$: Using the canonical identification of $\c{C}([0,1])\check{\ot}\b{M}_2$ with the C*-algebra of all continuous
functions from $[0,1]$ to $\b{M}_2$, let the morphism $\phi:\bb{C}^2\to\c{C}([0,1])\check{\ot}\b{M}_2$ be given by
$$[\phi(1,0)](x)=\left(\begin{array}{cc}x&\sqrt{x(1-x)}\\\sqrt{x(1-x)}&1-x \\\end{array}\right),\hspace{10mm}
[\phi(0,1)](x)=\r{I}_2-[\phi(1,0)](x),$$ and let $\sigma\in\c{S}(\b{M}_2)$ be defined by $(m_{ij})_{2\times2}\mapsto m_{11}$.
Then we have $\c{F}=(\r{id}\ot\sigma)\phi$.}\end{example}
We end this section by an important problem.
\begin{problem}Give an example of a non-implemented NFMO. In particular, give an example of a non-implemented FMO.\end{problem}
\section{Quantum Markov Chains}\label{2109180604}
Let $A$ be a C*-algebra. A \emph{quantum stochastic process} on $A$ in the sense defined by Accardi \cite{Accardi1,Accardi2,AccardiFrigerioLewis1}
is given by a triple $(B,\mu,(\psi_n)_{n\geq0})$ where $B$ is a C*-algebra, $\mu\in\c{S}(B)$, and where for each $n$,
$\psi_n:A\to B$ is a morphism.
\begin{definition}\label{2109150712}
Let $A$ be a C*-algebra, $\sigma\in\c{S}(A)$, and $(C_n,\phi_n,\nu_n)_{n\geq1}$
a sequence of RQMs on $A$. Let the C*-algebra $B:=A\check{\ot}\check{\ot}_{n=1}^\infty C_n$ be the direct limit of the direct system
$\{B_n\}_{n\geq0}$ where $B_n:=A\check{\ot} C_1\check{\ot}\cdots\check{\ot} C_n$, and let the state
$\mu:=\sigma\ot\ot_{n=1}^\infty\nu_n$ in $\c{S}(B)$ be the limit of
the sequence $(\mu_n)_n$ where $\mu_n\in\c{S}(B_n)$ is defined by $\mu_n:=\sigma\ot\nu_1\ot\cdots\ot\nu_n$.
Suppose that $\psi_n:A\to B$ denotes the composition $$\xymatrix{A\ar[rr]^{\phi_1\diamond\cdots\diamond\phi_n}&&B_n\ar[r]&B.}$$
The process $(B,\mu,(\psi_n)_{n\geq0})$ is called quantum Markov chain on
$A$ associated to $(C_n,\phi_n,\nu_n)_{n}$ with initial state $\sigma$. For a homogeneous QMC (i.e. the case that
all RQMs are identical) the index $n$ of $(C_n,\phi_n,\nu_n)$ is omitted.\end{definition}
The \emph{Markov property} considered in the following theorem is a dual version of (\ref{2109150710}). This is also of the type of Markov
properties that have been studied by Accardi in a series of papers. See \cite{AccardiSouissiSoueidy1} and references within.
\begin{theorem}
With the notations as in Definition \ref{2109150712}, suppose that $A_n$ denotes the C*-subalgebra $\psi_n(A)\subseteq B$
and $A_{n]}$ denotes the C*-subalgebra of $B$ generated by $A_0,\ldots,A_n$. The QMC $(B,\mu,(\psi_n)_{n\geq0})$ has the following Markov property:
For every $n\geq0$, there is an implemented NFMO $\b{E}_{n}:A_{n+1]}\to A_{n]}$ such that for $\alpha_n\in A_{n]},\alpha_{n+1}\in A_{n+1]}$ we have
$$\b{E}_{n}(\alpha_n\alpha_{n+1})=\alpha_n\b{E}_{n}(\alpha_{n+1}),\hspace{5mm}
\mu(\b{E}_{n}(\alpha_{n+1}))=\mu(\alpha_{n+1}),\hspace{5mm}\b{E}_{n}(A_{n+1})\subseteq A_n.$$\end{theorem}
\begin{proof}First of all note that $A_{n]}\subseteq B_n$. For every $n$ let $\b{F}_n:B_{n+1}\to B_n$ be defined by
$$\b{F}_n(a\ot c_1\ot\cdots\ot c_n\ot c_{n+1})=\nu_{n+1}(c_{n+1})(a\ot c_1\ot\cdots\ot c_n).$$
Then $\b{F}_n$ is an implemented NFMO. For any $b_{n+1}\in B_{n+1}$ and $b_n\in B_n\subset B_{n+1}$ we have
$\b{F}_n(b_nb_{n+1})=b_n\b{F}_n(b_{n+1})$ and $\mu(b_{n+1})=(\sigma\ot\nu_1\ot\cdots\ot\nu_{n+1})b_{n+1}=
\mu(\b{F}_n(b_{n+1}))$. We define $\b{E}_{n}$ to be the restriction of $\b{F}_n$
to $A_{n+1]}$. Thus, in order to complete the proof, it is enough to show that
$\b{F}_n(A_{n+1})\subseteq A_n$ and  $\b{F}_n(A_{n+1]})\subseteq A_{n]}$.
The first inclusion follows from the following observation: For any $a_{n+1}\in A_{n+1}$ there is $a\in A$ such that $a_{n+1}=\psi_{n+1}(a)$; thus
\begin{equation*}\begin{split}\b{F}_n(a_{n+1})&=\b{F}_n(\psi_{n+1}(a))\\
&=\b{F}_n([(\phi_1\diamond\cdots\diamond\phi_n)\ot\r{id}_{C_{n+1}}]\phi_{n+1}(a))\\
&=[(\phi_1\diamond\cdots\diamond\phi_n)\ot\nu_{n+1}]\phi_{n+1}(a)\\
&=(\phi_1\diamond\cdots\diamond\phi_n)[(\r{id}_A\ot\nu_{n+1})\phi_{n+1}(a)]\\
&=\psi_n(a')\in A_n,\end{split}\end{equation*}
where $A\ni a'=(\r{id}_A\ot\nu_{n+1})\phi_{n+1}(a)$. The second inclusion follows from the first one and the fact that $\b{F}_n$ has no effect
on the component $A\check{\ot} C_1\check{\ot}\cdots\check{\ot} C_n$ of $B_{n+1}$.\end{proof}
Let $\star$ denote \emph{free product} of C*-algebras i.e. coproduct in the category of C*-algebras.
Let $(B,\mu,(\psi_n)_{n\geq0})$ be a quantum process on a C*-algebra $A$. Consider the C*-algebra $\star_{n=0}^\infty A$ and the state
$\hat{\mu}$ on it defined by $\hat{\mu}:=\mu\Psi$ where $\Psi:\star_{n=0}^\infty A\to B$ denotes the limit of the canonical morphisms
$\psi_0\star\cdots\star\psi_n:\star_{i=0}^nA\to B$. Then the pair $(\star_{n=0}^\infty A,\hat{\mu})$ may be called
\emph{path-space description} of the process. Also for any finite ordered sequence $s_1<\ldots<s_k$ of non-negative integers the state
$\hat{\mu}_{s_1,\ldots,s_k}\in\c{S}(\star_{i=1}^k A)$ defined by $$\hat{\mu}_{s_1,\ldots,s_k}:=\mu(\psi_{s_1}\star\cdots\star\psi_{s_k})$$
may be called a \emph{finite-dimensional distribution} for the process. We call the process \emph{semi-commutative} if the state $\hat{\mu}$
factors through the obvious morphism $\hat{\r{id}}:\star_{n=0}^\infty A\to\check{\ot}_{n=0}^\infty A$ i.e. there exists a (necessarily unique) state
$\tilde{\mu}\in\c{S}(\check{\ot}_{n=0}^\infty A)$ such that $\hat{\mu}=\tilde{\mu}\hat{\r{id}}$. It can be checked that if for every $a,a'\in A$
and every $n,n'\geq0$ with $n\neq n'$ the elements $\psi_n(a)$ and $\psi_{n'}(a')$ commute in $B$ then the process is semi-commutative.
The QMCs considered in \cite[$\S$7]{Accardi0} are semi-commutative. Our QMCs in general are not semi-commutative.
\begin{example}\emph{Recall that a \emph{compact quantum semigroup} is a pair $(A,\phi)$ where $A$ is a C*-algebra
and $\phi:A\to A\check{\ot} A$ is a coassociative morphism. Thus any state $\nu\in\c{S}(A)$ gives rise to a RQM $(A,\phi,\nu)$ on $A$.
The associated QMCs are just new decorations for \emph{quantum random walks} studied by some authors. See for instance
\cite{LindsaySkalski1} and \cite{Baraquin1}.}\end{example}
\section{Invariant States and Stationary Processes}\label{2109180605}
In this section we consider the noncommutative analogues of some very well-known results about invariant probability measures
and stationary Markov chains.

Let $(C,\phi,\nu)$ be a RQM on a C*-algebra $A$. A state $\sigma\in\c{S}(A)$ is called \emph{invariant} if
$\c{T}_{\phi,\nu}(\sigma)=\sigma$. We denote the set of invariant states by $\c{I}_{\phi,\nu}(A)$.
\begin{theorem}\emph{$\c{I}_{\phi,\nu}(A)$ is a nonempty closed convex subset of $\c{S}(A)$.}\end{theorem}
\begin{proof}The existence of an invariant state follows from Markov-Kakutani's Fixed Point Theorem (\cite[Theorem I.3.3.1]{GranasDugundji1}).
The other properties of $\c{I}_{\phi,\nu}(A)$ are easily verified.\end{proof}
Let $(B,\mu,(\psi_n)_{n\geq0})$ be a quantum process on $A$. We say that the process is \emph{stationary} if its finite-dimensional
distributions are invariant under translations, that is for every ordered sequence $s_1<\ldots<s_k$ of non-negative integers
and any $\ell\geq0$ we have
\begin{equation}\label{2109150803}
\hat{\mu}_{s_1+\ell,\ldots,s_k+\ell}=\hat{\mu}_{s_1,\ldots,s_k}\end{equation}
An equivalent (and effective) definition has been given in \cite{AccardiFrigerioLewis1}: The process is called stationary if for any $r$-tuple
$(t_1,\ldots,t_r)$ of non-negative integers (not necessarily ordered and such that it is possible $t_i=t_j$ for $i\neq j$),
any $r$-tuple $(a_1,\ldots,a_r)$ of elements of $A$, and every positive integer $\ell$ we have
\begin{equation}\label{2109150804}
\mu(\psi_{t_1}(a_1)\cdots\psi_{t_r}(a_r))=\mu(\psi_{t_1+\ell}(a_1)\cdots\psi_{t_r+\ell}(a_r)).
\end{equation}
(Note that (\ref{2109150803}) is satisfied iff (\ref{2109150804}) is satisfied with every $r\geq1$ and $t_i\in\{s_1,\ldots,s_k\}$.)
\begin{theorem}
Let $(C,\phi,\nu)$ be a RQM on $A$ and let $(B,\mu,(\psi_n)_{n\geq0})$ denote the associated homogenous QMC with initial state
$\sigma\in\c{S}(A)$. Then the QMC is stationary iff $\sigma\in\c{I}_{\phi,\nu}(A)$.\end{theorem}
\begin{proof}Suppose that $\sigma\in\c{I}_{\phi,\nu}(A)$.
With notations as in Definition \ref{2109150712} and  (\ref{2109150804}), assume that $k\geq t_i$ for $i=1,\ldots,r$, and consider
$B_{t_i}\supseteq\phi^{\diamond^{t_i}}(A)$ canonically as a subspace of $B_k=A\check{\ot}\check{\ot}_{i=1}^kC$.
Note that $\sigma$ is also an invariant state for RQM
$(\check{\ot}_{i=1}^\ell C,\phi^{\diamond^{\ell}},\ot_{i=1}^\ell\nu)$ on $A$. We have
\begin{equation*}\begin{split}
&\mu(\psi_{t_1+\ell}(a_1)\cdots\psi_{t_r+\ell}(a_r))\\
=&\Big(\sigma\ot(\ot_{i=1}^\ell\nu)\ot(\ot_{i=1}^k\nu)\Big)\Big(\big[\phi^{\diamond^{\ell}}\ot\r{id}_{\ot_{i=1}^kC}\big]
\big(\phi^{\diamond^{t_1}}(a_1)\cdots\phi^{\diamond^{t_r}}(a_r)\big)\Big)\\
=&\big(\sigma\ot\ot_{i=1}^k\nu\big)\big(\phi^{\diamond^{t_1}}(a_1)\cdots\phi^{\diamond^{t_r}}(a_r)\big)\\
=&\mu(\psi_{t_1}(a_1)\cdots\psi_{t_r}(a_r)).\end{split}\end{equation*}
The `only if' part follows from (\ref{2109150804}) by $t_1=0$ and $r,\ell=1$.\end{proof}
Let $\phi:A\to A\check{\ot} C$ be a QFM on $A$. The \emph{skew product} associated to $\phi$ is a morphism $\phi^\dag$ on the
C*-algebra $B:=A\check{\ot}\check{\ot}_{n=1}^\infty C$, defined by $$\phi^\dag:B\to B,\hspace{10mm}\big(a\ot c_1\ot c_2\ot\cdots\big)
\mapsto\big(\phi(a)\ot c_1\ot c_2\ot\cdots\big).$$
\begin{theorem}\label{2109230500}
Let $(C,\phi,\nu)$ be a RQM on $A$. Then for any state $\sigma$ of $A$ we have $\sigma\in\c{I}_{\phi,\nu}$ iff
$\mu:=\sigma\ot\ot_{n=1}^\infty\nu$ is an invariant state for $\phi^\dag$ i.e. $\mu\phi^\dag=\mu$.\end{theorem}
\begin{proof}It follows from the identity $\mu\phi^\dag=(\c{T}_{\phi,\nu}\sigma)\ot\ot_{n=1}^\infty\nu$.\end{proof}
\begin{remark}
\emph{It follows from Theorem \ref{2109230500} that any RQM $(C,\phi,\nu)$ on $A$ with an invariant state $\sigma$
gives rise to a \emph{C*-dynamical system} $(\phi^\dag,\mu)$. Thus, similar with the classical case (e.g. \cite{Kifer2}),
it is possible to develop an ergodic theory for RQMs through the ergodic theory of C*-dynamical systems.}
\end{remark}

{\footnotesize\begin{thebibliography}{22}
\bibitem{Accardi0}
L. Accardi,
\emph{Noncommutative Markov chains},
in Proc. of Int. School of Mathematical Physics (1974): 268--295.
\bibitem{Accardi1}%
L. Accardi,
\emph{Nonrelativistic quantum mechanics as a noncommutative Markov process},
Advances in Mathematics 20 (1976): 329--366.
\bibitem{Accardi2}%
L. Accardi,
\emph{Quantum stochastic processes},
In Statistical Physics and Dynamical Systems, pp. 285--302. Birkh\"{a}user, Boston, MA, 1985.
\bibitem{AccardiFrigerioLewis1}%
L. Accardi, A. Frigerio, J.T. Lewis,
\emph{Quantum stochastic processes},
Publications of the Research Institute for Mathematical Sciences 18 no. 1 (1982): 97--133.
\bibitem{AccardiSouissiSoueidy1}%
L. Accardi, A. Souissi, E.G. Soueidy,
\emph{Quantum Markov chains: A unification approach},
Infinite Dimensional Analysis, Quantum Probability and Related Topics 23 no. 02 (2020): 2050016.
(arXiv:1811.00500 [math.OA])
\bibitem{Baraquin1}
I. Baraquin,
\emph{Random walks on finite quantum groups},
Journal of Theoretical Probability 33 no. 3 (2020): 1715--1736.
(arXiv:1812.06862 [math.QA])
\bibitem{BlumenthalCorson2}%
R.M. Blumenthal, H.H. Corson,
\emph{On continuous collections of measures},
In Proc. 6th Berkeley Sympos. Math. Statist. Probab, vol. 2, 33--40, 1972.
\bibitem{EffrosRuan1}
E.G. Effros, Z.-J. Ruan,
\emph{Operator spaces},
London Mathematical Society Monographs: New Series 23, Clarendon Press, Oxford, 2000.
\bibitem{GranasDugundji1}%
A. Granas, J. Dugundji,
\emph{Fixed point theory},
Springer Monographs in Mathematics, Springer-Verlag, New York, 2003.
\bibitem{Jost1}%
J. Jost, M. Kell, C.S. Rodrigues,
\emph{Representation of Markov chains by random maps: existence and regularity conditions},
Calculus of Variations and Partial Differential Equations 54 no. 3 (2015): 2637--2655.
(arXiv:1207.5003 [math.DS])
\bibitem{Kifer1}
Y. Kifer,
\emph{Random perturbations of dynamical systems},
Vol 16 Progress in Probability and Statistics, Birkh\"{a}user, Boston, 1988.
\bibitem{Kifer2}
Y. Kifer,
\emph{Ergodic theory of random transformations},
Vol. 10 Springer Science \& Business Media, 2012.
\bibitem{LindsaySkalski1}
J.M. Lindsay, A.G. Skalski,
\emph{Quantum random walk approximation on locally compact quantum groups},
Letters in Mathematical Physics 103 no.7 (2013): 765--775.
(arXiv:1110.3990 [math.OA])
\bibitem{Phillips1}
N.C. Phillips,
\emph{Inverse limits of C*-algebras and applications}, in
Operator algebras and applications vol. 1, London Mathematical Society Lecture Note Series 135,
Cambridge University Press (1988): 127--185.
\bibitem{Rebowski1}%
R. Rebowski,
\emph{A note on integral representation of Feller kernels},
Annales Polonici Mathematici 56 (1991): 93--96.
\bibitem{Sadr1}%
M.M. Sadr,
\emph{On the quantum groups and semigroups of maps between noncommutative spaces},
Czechoslovak Math. J. 67 no. 1 (2017): 97--121. (arXiv:1506.06518 [math.QA])
\bibitem{Sadr3}%
M.M. Sadr,
\emph{Path-connected components of affine schemes and algebraic K-theory}, preprint.
(arXiv:1911.04204 [math.KT])
\bibitem{Skoufranis1}
P. Skoufranis,
\emph{Completely positive maps},
available online at pskoufra.info.yorku.ca, (2014).
\bibitem{Soltan1}%
P.M. So{\l}tan,
\emph{Quantum families of maps and quantum semigroups on finite quantum spaces},
J. Geom. Phys. 59 (2009): 354--368.
(arXiv:math/0610922 [math.OA])
\end{thebibliography}}
\end{document}